\documentclass[10pt]{article}
\usepackage{graphicx,caption}
\usepackage{latexsym,booktabs}
\usepackage{blindtext,textcomp}
\usepackage{amsthm, amsmath, amssymb, amsbsy}
\usepackage{hyperref}
\usepackage{tikz}
\usepackage[utf8]{inputenc}

\theoremstyle{definition}
\newtheorem{theorem}{Theorem}[section]
\newtheorem{remark}[theorem]{Remark}
\newtheorem{cor}[theorem]{Corollary}
\newtheorem{pro}[theorem]{Proposition}
\newtheorem{lem}[theorem]{Lemma}

\def \al {{\alpha}}

\def \R {{\mathbb R}}

\def \C {{\mathbb C}}

\def \Ll {{\mathcal L}}

\def \Ll {{\mathcal L}}

\def \la {{\lambda}}

  \author{Priyanka Grover \thanks{The research of this author is supported by INSPIRE Faculty Award IFA-14-MA52 of DST, India, and by Early Career Research Award ECR/2018/001784 of SERB, India. Email: priyanka.grover@snu.edu.in}, Veer Singh Panwar \thanks{Email: vs728@snu.edu.in}, A. Satyanarayana Reddy\thanks{Email: satyanarayana.reddy@snu.edu.in}\\
Department of Mathematics\\Shiv Nadar University, Dadri\\ U.P. 201314, India.}
  \date{}
  
\begin{document}
	\title{Positivity properties of some special matrices}
		\maketitle
 \begin{center}
 \large{\textbf{Abstract}}
 \end{center}
	\begin{footnotesize}
 It is shown that for positive real numbers $ 0<\lambda_{1}<\dots<\lambda_{n} $, $ \left[\frac{1}{\beta({\la_i},{\la_j})}\right]$, where $ \beta(\cdot,\cdot)$ denotes the beta function, is infinitely divisible and totally positive. For $ \left[\frac{1}{\beta({i},{j})}\right]$, the Cholesky decomposition and successive elementary bidiagonal decomposition are computed.  Let $\mathfrak w(n)$ be the $n$th Bell number. It is proved that  $\left[\mathfrak w(i+j)\right]$ is a totally positive matrix but is infinitely divisible only upto order $4$. It is also shown that the symmetrized Stirling matrices are totally positive.  
	\end{footnotesize} \\ 
	
	
	\begin{footnotesize}
	\end{footnotesize}
		
	\section{Introduction}
	Let $ M_n(\C) $ be the set  of all $ n\times n $ complex matrices. A matrix $A\in M_n(\C)$ is said to be {\em positive semidefinite} if 	$x^*Ax \geq 0$ for all $x \in \C^n$ and {\em positive definite}  if $x^*Ax > 0$ for all $x \in \C^n, x\neq 0.$	Let $A=[a_{ij}]$ and $B=[b_{ij}]$. In this paper, $1\leq i,j\leq n$, unless otherwise specified. The {\em Hadamard product} or  {\em the Schur product}  of two matrices $ A $ and $ B $ is denoted by $A\circ B,$ where $ A\circ B = [a_{ij}b_{ij}].$  For a nonnegative real number $r$,   $A^{\circ r}= [a_{ij}^r].$  

	Let $A=[a_{ij}]$ be  such that $a_{ij}\ge 0$.  The matrix $A$ is called {\em infinitely divisible} if $A^{\circ r}$ is positive semidefinite for every real number $r>0$.	
	For examples and properties of infinitely divisible matrices, see~\cite{R.Bhatia,Mean Matrices, horn}.  The  matrix $A$ is called  {\em totally positive} or  {\em totally nonnegative} if all its minors are positive or nonnegative respectively. For more results on these, see~\cite{fallat}.  The main objective of this paper is to explore the above mentioned properties for a few matrices which are constructed from interesting functions. Many such matrices have been studied in \cite{R.Bhatia, Mean Matrices, bhatia2, bhatiajain}. A famous example of such a matrix is the Hilbert matrix $\left[\frac{1}{i+j-1}\right]$. Another important example is the Pascal matrix $\mathcal P=\left[\binom{i+j}{i}\right]_{i,j=0}^{n}$. Both of these are known to be infinitely divisible and totally positive. In \cite{strang}, \emph{Cholesky decomposition} of $\mathcal P$ was given, that is, a lower triangular matrix $L$  was obtained such that $\mathcal P=LL^*$.

Let $\mathcal{B}=[\frac{1}{\beta(
{i},
{j})}]$, where  $ \beta(\cdot,\cdot) $ is the \textit{beta function}. We call $ \mathcal{B} $ as the \textit{beta matrix}. By definition,  
$${\beta(i,j)} = \frac{\Gamma (i) \Gamma (j)}{\Gamma(i+j)},$$ where $\Gamma(\cdot)$ is the Gamma function. Thus
$$\mathcal B= \left[\frac{(i+j-1)!}{(i-1)!(j-1)!}\right].$$ Note that the entries of this matrix look similar to those of $\mathcal P$. 
Using the infinite divisibility and total positivity of $\mathcal P$, we  show that $\mathcal B$ is infinitely divisible and totally positive, respectively. 
We also compute the {Cholesky decomposition} of $\mathcal B$. 

Let $E_{i,j}$ be the matrix whose $(i,j)$th entry is $1$ and others are zero. For any complex numbers $ s,t, $ let\begin{align*}
L_{i}(s) = I+sE_{i,i-1} \,\, \text{and}\,\, U_{j}(t) = I + tE_{j-1,j},
\end{align*}
where $ 2 \leq i,j \leq n $. Matrices of the form $ L_{i}(s) $ or $ U_{j}(t) $  are called \textit{ elementary bidiagonal matrices}. For a vector $(d_1,d_2,\ldots,d_n)$, let $\text{diag}([d_i])$ denote the diagonal matrix $\text{diag}(d_1 ,\ldots, d_n)$.
An $n\times n$ matrix is totally positive if and only if it can be written as $(L_n\left(l_{k}\right)L_{n-1}\left(l_{k-1}\right)\cdots L_2\left(l_{k-n+2}\right))\\ (L_n\left(l_{k-n+1}\right)L_{n-1}\left(l_{k-n}\right) \cdots L_3\left(l_{k-2n+4}\right))\cdots(L_n\left(l_{1}\right))D (U_n\left(u_{1}\right)) (U_{n-1}\left(u_{2}\right) U_n\left(u_{3}\right))\cdots (U_2\left(u_{k-n+2}\right)\cdots \\ U_{n-1}\left(u_{k-1}\right)U_n\left(u_{k}\right)),$ where $k=\binom{n}{2},\  l_i,u_j>0$ for all $i,j\in\{1,2,\ldots,k\}$ and $D=\text{diag}[(d_i)]$ is a diagonal matrix with all $d_i > 0$ \cite[Corollary 2.2.3]{fallat}. This particular factorization (with $l_i,u_j\geq 0$) is called \textit{successive elementary bidiagonal (SEB) factorization} or \textit{Neville factorization}.  
 We obtain the SEB factorization for $ \left[\frac{1}{\beta(i,j)}\right] $ explicitly.
 
Let $(x)_{0}=1$ and for a positive integer $k$, let $(x)_{k} = x(x-1)(x-2)\cdots (x-k+1)$.
The {\it Stirling numbers of first kind} $s(n,k)$ ~\cite[p. 213]{comtet} and the {\it Stirling numbers of second kind} $S(n,k)$ ~\cite[p. 207]{comtet} are respectively defined as $$(x)_{n} = \sum_{k=0}^{n}s(n,k)x^{k}$$
 and  $$x^n= \sum_{k=0}^{n}S(n,k)(x)_{k}.$$ 
 The {\it unsigned Stirling matrix of first kind} $\mathfrak{s}=[\mathfrak{s}_{ij}]$ is defined as 
 $$\mathfrak s_{ij}=\begin{cases}
    (-1)^{i-j}s(i,j) & \text{if $i\geq j$},\\
    0 & \text{otherwise}.
  \end{cases}$$ 
  The {\it Stirling matrix of second kind}  $\mathcal{S}=[\mathcal{S}_{ij}]$ is defined as 
  $$\mathcal S_{ij}=\begin{cases}
    S(i,j) & \text{if $i\geq j$},\\
    0 & \text{otherwise}.
  \end{cases}$$
  It is a well known fact that $\mathfrak{s}$ and $\mathcal S$ are totally nonnegative (see for example \cite{galvin}). We consider the matrices $\mathfrak{s} \mathfrak{s}^*$ and $\mathcal S \mathcal S^*$ and call them {\it  symmetrized unsigned Stirling matrix of first kind} and {\it symmetrized Stirling matrix of second kind}, respectively. By definition, these are positive semidefinite and totally nonnegative. We show that both these matrices are in fact totally positive.

Another matrix that we consider is formed by the well known {\em Bell numbers}. 
The sum $\mathfrak {w}({n}) = \sum_{k=0}^{n}S(n,k)$ is the number of partitions of a set of $n$ objects and is known as the {$n$th Bell number} \cite[p. 210]{comtet}. Consider the matrix $\mathfrak B=[\mathfrak {w}(i+j)]_{i,j=0}^{n-1}$. Let $X=(x_{ij})_{i,j=0}^{n-1}$ be the lower triangular  matrix defined recursively by
\begin{align*}x_{00} &= 1,\ x_{0j}= 0 \text{ for } j>0,\ \text{ and }x_{ij} = x_{i-1,j-1}+(j+1)x_{i-1,j}+(j+1)x_{i-1,j+1} \text{ for } i\geq 1.\end{align*} (Here $x_{i,-1}=0$ for every $i$ and $x_{in}=0$ for $i=0,\ldots, n-2$.)
This is  known as the Bell triangle \cite{Xi chen}. 
Lemma 2 in \cite{Martin} gives the Cholesky decomposition of $\mathfrak B$ as $LL^*$, where $L=X \ \text{diag}\left( \left[\sqrt{i!}\right]\right)_{i=0}^{n-1}$. It is also shown in \cite{Martin} that $\det \mathfrak B=\prod_{i=0}^{n-1}  i!$. It is a known fact that $\mathfrak B$ is totally nonnegative \cite{Liang}. We show that $\mathfrak B$ is totally positive. We also show that $\mathfrak B$ is infinitely divisible only upto order $4$.


 
 In Section 2 we give our results for the beta matrix, namely, its infinite divisibility and total positivity, its Cholesky decomposition, its determinant, ${\mathcal B}^{-1}$, and its SEB factorization.  We show that $\mathcal B^{\circ r}$ is in fact totally positive for all $r>0$. We end the discussion on the beta matrix by proving  that for positive real numbers $\lambda_1< \cdots< \lambda_n$ and  $\mu_1<\cdots<\mu_n$, $\left[\frac{1}{\beta(\lambda_i,\lambda_j)}\right]$ is an infinitely divisible matrix and  $\left[\frac{1}{\beta(\lambda_i,\mu_j)}\right]$ is a totally positive matrix.
 In section $ 3 $, we prove that symmetrized Stirling matrices and $\mathfrak B$ are totally positive. For the first kind, we give the SEB factorization for  $\mathfrak{s}$. For the second kind, we show that $\mathcal S$ is {\em triangular totally positive} \cite[p. 3]{fallat}. We also show that $\mathfrak B$ is infinitely divisible if and only if its order is less than or equal to  $4. $

\section{The beta matrix}\label{section:Beta}
The  infinite divisibility and total positivity of $\mathcal B$ are easy consequences of the corresponding results for $\mathcal P$.   
For $1\leq i,j\leq n$, let $A(i,j)$ denote the submatrix of $A$ obtained by deleting $i$th row and $j$th column from $A$. 
 Each $A(i,i)$ is infinitely divisible, if $A$ is infinitely divisible, and each $A(i,j)$ is totally positive, if $A$ is totally positive.
\begin{theorem}\label{thm:2}
The matrix $\mathcal B=\left[\frac{1}{\beta(i,j)}\right]$ is  infinitely divisible and totally positive.
\end{theorem}
\begin{proof}

By definition,  
\begin{equation}\frac{1}{\beta(i,j)}= \frac{(i+j-1)!}{(i-1)!(j-1)!}
			= \frac{ij(i+j)!}{(i+j)i!j!}
			=\frac{1}{\frac{1}{i}+\frac{1}{j}}.\binom{i+j}{i}.\label{eq1.1}\end{equation}
Thus $\mathcal B= C \circ \mathcal P(1,1),$
			where $ C = \left[\frac{1}{\frac{1}{i}+\frac{1}{j}}\right] $ is a Cauchy matrix. 
Both $C$ and $\mathcal P(1,1)$  are  infinitely divisible $\cite{R.Bhatia}$. Since Hadamard product of infinitely divisible matrices is infinitely divisible, we get $\mathcal B$ is infinitely divisible.

Again, note that 
\begin{equation}\frac{1}{\beta(i,j)} =\frac{(i+j-1)!}{(i-1)!(j-1)!}= j\frac{(i+j-1)!}{(i-1)!(j!)} = \binom{(i-1)+j}{i-1} j \label{eqn : Beta is Tp}.\end{equation}
So $\mathcal B$ is the product of the totally positive matrix $\mathcal P(n+1,1)$ with the positive diagonal matrix $\text{diag}([i])$. Hence $\mathcal B$ is totally positive.
\end{proof}

The below remarks were suggested by the anonymous referee.
\begin{remark}
Since for each $r>0$, $C^{\circ r}$ and $\mathcal P^{\circ r}$ are positive definite (see p. 183 of \cite{pdm}), \eqref{eq1.1} gives that $\mathcal B^{\circ r}$ is positive definite. Let $\mathcal G=\left[(i+j)!\right]_{i,j=0}^n$. Then $\mathcal G$ is a Hankel matrix. 
Since $\mathcal G^{\circ r}$ is congruent to $\mathcal P^{\circ r}$ via the positive diagonal matrix $\text{diag}([i!^r])_{i=0}^n $
,  $\mathcal G^{\circ r}$ is positive definite. The matrix $\mathcal G^{\circ r}(n+1,1)=\left[(i+j-1)!^r\right]$ is congruent to $\mathcal B^{\circ r}$, 
 via the positive diagonal matrix $\text{diag}\left[{(i-1)!^r}\right]$. So $\mathcal G^{\circ r}(n+1,1)$ is also positive definite. Hence $\mathcal G^{\circ r}$ is totally positive, by Theorem 4.4 in \cite{pinkus}. This shows that $\mathcal P^{\circ r}$ is totally positive. By \eqref{eqn : Beta is Tp}, $\mathcal B^{\circ r}$ is also totally positive. 
 \end{remark}
\begin{remark}Another proof for infinite divisibility of $\mathcal B$ can be given as follows.  For positive real numbers $\la_{1}, \la_{2}, \ldots ,\la_{n}$,  the generalized Pascal matrix is the matrix $\left[\frac{\Gamma(\lambda_{i}+\lambda_{j}+1)}{\Gamma(\lambda_{i}+1)\Gamma(\lambda_{j}+1)}\right]$. The beta matrix $\mathcal B$ is congruent to it via a positive diagonal matrix, when $\la_i=i-1/2$.  The generalized Pascal matrix is infinitely divisible \cite{R.Bhatia},  and hence so is $\mathcal B$.\end{remark}

By Theorem \ref{thm:2}, $\mathcal B$ is positive definite and so, can be written as $LL^*$. Our next theorem gives $L$ explicitly.

\begin{theorem}\label{thm:Beta:CD:I}
The matrix $[\frac{1}{\beta(i,j)}]$ has the Cholesky decomposition $LL^{*}$, where $L = [\binom{i}{j}\sqrt{j}]$.
\end{theorem}

\begin{proof}
 We prove the result by the combinatorial method of  two way counting. Consider a set of $ (i+j-1) $ persons. The number of ways of choosing a committee of $j$ people and a chairman of this committee is  
$$j\cdot \binom{i+j-1}{j} = \frac{(i+j-1)!}{(i-1)!(j-1!)}.$$
Another way to count the same is to separate $ (i+j-1) $ people into two groups of $ i $ and $ (j-1) $ people each. The number of ways of choosing a committee of $ j $ people from these two groups of people and then a chairman  of the committee is $j\cdot\sum_{k}^{}  \binom{i}{k}\binom{j-1}{j-k}$, where $k$ varies from $1$ to min$\{{i}, {j}\}$. Rearranging the terms in this expression, we get
\begin{align*}	 
j\cdot\sum_{k}^{}  \binom{i}{k}\binom{j-1}{j-k} &=
\sum_{k}^{}j\binom{i}{k}\binom{j-1}{j-k}\\
&=\sum_{k}^{}\binom{i}{k}j\frac{(j-1)!}{(j-k)!(k-1)!}\\
&=\sum_{k}^{}\binom{i}{k} {k}\frac{j!}{(j-k)!k!}\\
&=\sum_{k}^{}k\binom{i}{k}\binom{j}{k}.\\
\end{align*}The last expression is the $(i,j)$th entry of the matrix $LL^{*}$, where $L =[\binom{{i}}{j}\sqrt{j}]$. 
\end{proof}



\begin{cor}
The determinant of $\mathcal B$ is equal to $ n!$.
\end{cor}

\begin{cor}
	The inverse of the matrix $\mathcal B$ has $(i,j)$th entry as $(-1)^{i+j}\sum_{k=1}^{n}\binom{k}{i}\binom{k}{j}\frac{1}{k}. $ 
\end{cor}

\begin{proof}
By Theorem \ref{thm:Beta:CD:I}, $\mathcal B = LL^{*}$, where $L=\left[\binom{i}{j}\sqrt{j}\right]$. Let $  Z = [\binom{i}{j}]$ and $D'=\text{diag}([\sqrt{i}])$. Then $L=ZD'$. Since $\sum_{k=1}^{n}(-1)^{k+j}\binom{i}{k}\binom{k}{j} = \delta_{ij}$, we get that $Z^{-1} = [(-1)^{i+j}\binom{i}{j}] $. So $ L^{-1}= D'^{-1}Z^{-1}=\left[(-1)^{i+j}\binom{i}{j}\frac{1}{\sqrt{i}}\right].$ Thus $\mathcal B^{-1}=L^{*^{-1}} L^{-1}$. 
This gives that the $(i,j)$th entry of $\mathcal B^{-1}$ is $$\sum_{k=1}^{n}\left((-1)^{k+i}\binom{k}{i}\frac{1}{\sqrt{k}}\right)\left((-1)^{k+j}\binom{k}{j}\frac{1}{\sqrt{k}}\right)$$ which is equal to
$$(-1)^{i+j}\sum_{k=1}^{n}\binom{k}{i}\binom{k}{j}\frac{1}{k}.$$
\end{proof}


\begin{remark}  The above theorems hold true if  $i,j$ are replaced by  $\lambda_i, \lambda_j$, where $ 0 < \lambda_{1}< \cdots <\lambda_{n} $ are positive integers. For $r>0$, the matrix $\left[\frac{1}{\beta(\la_{i},\la_{j})^r}\right]$ is  totally positive because it is a submatrix of the $\lambda_n \times \lambda_n$ matrix $\mathcal B^{\circ r}$. So $\left[\frac{1}{\beta(\la_{i},\la_{j})}\right]$ is also infinitely divisible. We have $\left[\frac{1}{\beta(\la_{i},\la_{j})}\right]=LL^*$, where $ L$ is the $n \times \la_{n}$ matrix $[\binom{\lambda_{i}}{j}\sqrt{j}]$. The proof is same as for Theorem \ref{thm:Beta:CD:I}. 
\end{remark}
The next theorem gives the SEB factorization for the matrix $\mathcal B$, which also gives another proof for $\mathcal B$ to be totally positive.

\begin{theorem}
\label{beta BD}
	The matrix  $\mathcal B =\left[\frac{1}{\beta(i,j)}\right]$ can be written as
	\begin{align*}
	 \left(L_{n}\left(\frac{n}{n-1}\right)L_{n-1}\left(\frac{n-1}{n-2}\right)\cdots L_{2}\left(2\right)\right)\left(L_{n}\left(\frac{n}{n-1}\right)\cdots L_{3}\left(\frac{3}{2}\right)\right)\ldots \left(L_{n}\left(\frac{n}{n-1}\right)\right)D\\ \left(U_{n}\left(\frac{n}{n-1}\right)\right)\ldots\left(U_{3}\left(\frac{3}{2}\right)\ldots U_{n}\left(\frac{n}{n-1}\right)\right)\left(U_{2}\left(2\right)\cdots U_{n-1}\left(\frac{n-1}{n-2}\right) U_{n}\left(\frac{n}{n-1}\right) \right),
	\end{align*}
	where $D=\text{diag}([i]).$
\end{theorem}
To prove this theorem, we first need a lemma.
\begin{lem}
For $1\leq k\leq n-1$, let $Y_k=\left[y^{(k)}_{ij}\right]$ be the $n\times n$ lower triangular matrix where
	$$y^{(k)}_{ij} = 
	\begin{cases}
	\frac{i}{j} \binom{i-(n-k)}{j-(n-k)}   & \,\, \text{if}\; \,\,   n-k\leq j<i\leq n,\\
	1 & \,\, \text{if}\; i=j,\\
	0 & \,\, \text{otherwise}
	.
	\end{cases}$$
	Then
	\begin{equation}
	    Y_k = \left(L_{n}\left(\frac{n}{n-1}\right)\cdots L_{n-(k-1)}\left(\frac{n-k+1}{n-k}\right)\right)\left(L_{n}\left(\frac{n}{n-1}\right)\cdots L_{n-(k-2)}\left(\frac{n-k+2}{n-k+1}\right)\right)\cdots\left(L_{n}\left(\frac{n}{n-1}\right)\right).\label{eq:Yk} 
	\end{equation}
\end{lem}
\begin{proof}
 For $k=1$, the right hand side is $L_{n}\left(\frac{n}{n-1}\right)$, which is same as $Y_1$.  We show below that for $1\leq k\leq n-2$, 
	\begin{equation}
(\Ll_n\Ll_{n-1}\cdots\Ll_{n-k})Y_k=Y_{k+1},\label{induction}
	\end{equation}
where $\Ll_p$ denotes $L_p\left(\frac{p}{p-1}\right)$.
This will show that \eqref{eq:Yk} is true for $k=1,2,\ldots, n-1$.
We also keep note of the fact that multiplying $\Ll_p$ on the left of a matrix $A$ is applying the elementary row operation  $
        \text{row }{p} \rightarrow \text{row } {p}+\left(\frac{p}{p-1}\right)\times \text{row } ({p-1}) \label{row}
    $ on $A$, which we will use in the cases $2,3$ and $4$. We also note that for $k=n-2$, only cases $1$ and $4$ are relevant.
    
Case $1$: Let $1\leq i\leq j\leq n$. Since $\mathcal L_p$ and $Y_k$ are lower triangular matrices with diagonal entries 1, so is $(\Ll_n\Ll_{n-1}\cdots\Ll_{n-k})Y_k$. So the $(i,j)$th entry of both the matrices in \eqref{induction} is same.

Case $2$: Let $1\leq j\leq n-k-2$ and $j+1\leq i\leq n-k-1$. In this case, $y^{(k)}_{ij}=0$. Since multiplying $Y_k$ on the left by $\Ll_n\Ll_{n-1}\cdots\Ll_{n-k}$  will keep its rows $j+1,\ldots, n-k-1$ unchanged, we get that the $(i,j)$th entry of $(\Ll_n\Ll_{n-1}\cdots\Ll_{n-k})Y_k$ is zero. 
 
 Case $3$: Let $1\leq j\leq n-k-2$ and $ n-k \leq i\leq n$. Multiplying $Y_k$ on the left by $\Ll_{n-k},\Ll_{n-(k-1)},\ldots,\Ll_{n}$ successively, we get that the $(i,j)$th entry of $(\Ll_n\Ll_{n-1}\cdots\Ll_{n-k})Y_k$ is given by
 \begin{align}
& y^{(k)}_{ij}+\frac{i}{i-1}\left[y^{(k)}_{i-1,j}+\frac{i-1}{i-2}\left[y^{(k)}_{i-2,j}+\cdots  + \frac{n-k+1}{n-k} \left[y^{(k)}_{n-k,j}+\frac{n-k}{n-k-1}y^{(k)}_{n-k-1,j}\right]\right] \cdots \right]\nonumber\\
     &=y^{(k)}_{ij}+\frac{i}{i-1}y^{(k)}_{i-1,j}+\frac{i}{i-2}y^{(k)}_{i-2,j}+\cdots + \frac{i}{n-k-1}y^{(k)}_{n-k-1,j}.\label{ij entry}
 \end{align}
Now  $y^{(k)}_{pj}=0$ for all $1\leq j\leq n-k-2$ and $p\neq  j$. So the $(i,j)$th entry of $(\Ll_n\Ll_{n-1}\cdots\Ll_{n-k})Y_k$ is zero. 
 
 Case $4$: Let $i>j\geq n-k-1$. 
Again, the $(i,j)$th entry of $(\Ll_n\Ll_{n-1}\cdots\Ll_{n-k})Y_k$ is given by $y^{(k)}_{ij}+\frac{i}{i-1}y^{(k)}_{i-1,j}+\frac{i}{i-2}y^{(k)}_{i-2,j}+\cdots + \frac{i}{j}y^{(k)}_{jj}+\cdots+ \frac{i}{n-k-1}y^{(k)}_{n-k-1,j}.$ For $j=n-k-1$, $y^{(k)}_{pj}=0$ for $p\neq j$. So  the $(i,n-k-1)$th entry of $(\Ll_n\Ll_{n-1}\cdots\Ll_{n-k})Y_k$ is $\frac{i}{n-k-1}$. For $j\geq n-k$,  $y^{(k)}_{pj}=0$ for $p<j$ and $y^{(k)}_{pj}=\frac{p}{j}\binom{p-n+k}{j-n+k}$ for $p\geq j$. So we obtain that the $(i,j)$th entry of $(\Ll_n\Ll_{n-1}\cdots\Ll_{n-k})Y_k$ is
     \begin{align}
     &=\frac{i}{j}\binom{i-n+k}{j-n+k} +\frac{i}{i-1}\cdot \frac{i-1}{j}\binom{(i-1)-n+k}{j-n+k}+\frac{i}{i-2}\cdot \frac{i-2}{j}\binom{(i-2)-n+k}{j-n+k}+\cdots +\frac{i}{j}\nonumber\\
     &=\frac{i}{j}\bigg[\sum_{p=j}^{i}\binom{p-n+k}{j-n+k}\bigg].\label{case4}
\end{align}
Since $\sum\limits_{k=0}^n\binom{m+k}{m}=\binom{m+n+1}{m+1}, $
	the expression in \eqref{case4} is equal to $\frac{i}{j}\binom{i-n+k+1}{j-n+k+1}.$
In all the above four cases, the $(i,j)$th entry of $(\Ll_n\Ll_{n-1}\cdots\Ll_{n-k})Y_k$ is the same as that of $Y_{k+1}$. Hence we are done.
\end{proof}

	\textit{Proof of Theorem \ref{beta BD}.} By Theorem \ref{thm:Beta:CD:I} we have that $\mathcal B=\left[\binom{i}{j}\right] \ D\  \left[\binom{i}{j}\right]^*$. 
So it is enough to show that \begin{equation}
(\Ll_n\Ll_{n-1}\cdots\Ll_2)(\Ll_n\Ll_{n-1}\cdots\Ll_3)\cdots (\Ll_n\Ll_{n-1})(\Ll_n)=\left[\binom{i}{j}\right].\label{eq:BD}\end{equation} 
This is easily obtained by putting $k = n-1$ in $\eqref{eq:Yk}$.

 \qed




 
\begin{remark} Let $p_{0}, \ldots, p_{n-1}$ be functions from a set $\mathfrak X$ to a field  and $\la_{1}, \ldots, \la_{m} \in \mathfrak X$. Then the  $m \times n$ matrix defined by $ \left[p_{j-1}(\la_{i})\right]_{1 \leq i \leq m, 1 \leq j \leq n}$ is called  an {\em alternant matrix} \cite[p. 112]{aitken}. Let $\langle x \rangle_{n}=x (x+1)\cdots (x+n-1)$.
For positive integers $ 0 < \lambda_{1}< \cdots < \lambda_{n} $, we have
\begin{equation*}
\frac{1}{\beta(\lambda_{i},\lambda_{j})} = \frac{(\lambda_{i}+\lambda_{j}-1)!}{(\lambda_{i}-1)!(\lambda_{j}-1)!}=\frac{\langle \lambda_j \rangle_{\lambda_i}}{(\lambda_{i}-1)!}.
\end{equation*}
Thus with  $p_{j-1}(x)= \frac{\langle \lambda_j \rangle_{x}}{(x-1)!} $,
$ \left[\frac{1}{\beta(\lambda_{i},\lambda_{j})}\right] $ is an alternant matrix.
\end{remark}

 We now consider the more general matrix $\left[\frac{1}{\beta(\la_{i},\la_{j})}\right]$ for positive real numbers $\la_{1}, \la_{2}, \ldots ,\la_{n} $. The $(i,j)$th entry of this matrix is given by $\frac{\Gamma(\lambda_{i}+\lambda_j)}{\Gamma(\lambda_{i})\Gamma(\lambda_{j})}$. The proof for infinite divisibility of generalized Pascal matrix $\left[\frac{\Gamma(\lambda_{i}+\lambda_{j}+1)}{\Gamma(\lambda_{i}+1)\Gamma(\lambda_{j}+1)}\right]$ is given in \cite{R.Bhatia}. Infinite divisibility of $\left[\frac{1}{\beta(\la_{i},\la_{j})}\right]$ follows by a similar argument. Alternatively, one can also observe that $\left[\frac{1}{\beta(\la_{i},\la_{j})}\right]=\left[\frac{\Gamma(\lambda_{i}+\lambda_{j}+1)}{\Gamma(\lambda_{i}+1)\Gamma(\lambda_{j}+1)}\right]\circ \left[\frac{1}{\frac{1}{\lambda_{i}}+\frac{1}{\lambda_{j}}}\right]$ and deduce its infinite divisibility. This is also same as saying that $\frac{1}{\beta(\cdot,\cdot)}$ is an {\it infinitely divisible kernel} \cite{horn} on $\mathbb R^+\times \mathbb R^+$. 
  
We observe that $\frac{1}{\beta(\cdot,\cdot)}$ is also a {\it totally positive kernel} \cite{karlin} on $\mathbb R^+\times \mathbb R^+$.  For that we first show that  $\left[\Gamma(\lambda_{i}+\mu_{j})\right]$ is totally positive, where $ 0<\lambda_{1}<\dots<\lambda_{n} $ and $ 0<\mu_{1}<\dots<\mu_{n} $. The proof of this was guided to us by Abdelmalek Abdesselam and Mateusz Kwa\'snicki \footnote{https://mathoverflow.net/questions/306366/}.


\begin{theorem}\label{theorem 2.6}
Let $ 0<\lambda_{1}<\dots<\lambda_{n} $ and $ 0<\mu_{1}<\dots<\mu_{n} $ be positive real numbers. Then $\left[\Gamma(\lambda_{i}+\mu_{j})\right]$ is totally positive. 
\end{theorem}
\begin{proof}

	Since all the minors of $\left[\Gamma(\lambda_{i}+\mu_{j})\right]$ are also of the same form, it is enough to show that $\det\left( \left[\Gamma(\lambda_{i}+\mu_{j})\right]\right)>0.$ 
Let	$ K_{1}(x,y) = x^{y}$ and $K_{2}(x,y) = y^x $. 
	For any $ \lambda, \mu \in \R^{+} $,
	\begin{eqnarray}
		\Gamma(\lambda+\mu)&=& \int\limits_{0}^{\infty} e^{-t}t^{\lambda+\mu-1} dt\nonumber\\
		&=& \int\limits_{0}^{\infty}t^{\lambda +\mu} \left(\frac{e^{-t}}{t}\right) dt\nonumber\\
		&=& \int\limits_{0}^{\infty}t^{\lambda} t^{\mu} \sigma(dt), \quad \text{where}\,\, \sigma(dt) = \frac{e^{-t}}{t}dt\nonumber\\
		&=&\int_{0}^{\infty}K_{2}(\lambda,t)K_{1}(t,\mu)\sigma(dt).\label{K2K1}
	\end{eqnarray}
	
	Let $0<t_{1}<\dots<t_{n}$.  Then by  \eqref{K2K1} and the \textit{basic composition formula}  \cite[p. 17]{karlin} , 
	\begin{align*}
	\det\left(\left[\Gamma(\lambda_{i}+\mu_{j})\right]\right)  &= 
	\int_{t_{1}=0}^{\infty}\ldots \int_{t_{n}=0}^{\infty} \det\left([K_{2}(\lambda_{i},t_{j})]\right) 
	\times 
	\det\left([K_{1}(t_{i},\mu_{j})]\right)
	\sigma(dt_{1})\cdots \sigma(dt_{n}).
	\end{align*}
	Since  $K_1$ and $K_2$ are totally positive kernels on $ \R^{+}\times \R^{+}$ (see \cite[p. 90]{pinkus}),
	$\det\left([K_{2}(\lambda_{i},t_{j})] \right)$
	and $
 \det\left([K_{1}(t_{i},\mu_{j})]\right) $ are positive functions of $ t_{1}, \ldots, t_{n} $.  So $ \det \left(\left[\Gamma(\lambda_{i}+\mu_{j})\right]\right) >0 $.
\end{proof}
\begin{cor} \label{TP}
For positive real numbers $0<\la_{1}< \la_{2}< \cdots <\la_{n} $ and $0<\mu_1<\mu_2<\cdots<\mu_n$, the matrix $\left[\frac{1}{\beta(\la_{i},\mu_{j})}\right]$ is  totally positive.\end{cor}
\begin{proof}

Since $\mathcal{B} = \text{diag}\left(\left[ \frac{1}{\Gamma(\lambda_{i})}\right]\right)\ \left[\Gamma(\lambda_i+\mu_j)\right] \text{diag}\left(\left[ \frac{1}{\Gamma(\mu_{i})}\right]\right)$, we obtain the required result.
\end{proof}
	
\section{Combinatorial matrices}\label{section:Bell}

The unsigned Stirling matrix of first kind $\mathfrak{s}$ is totally nonnegative as well as invertible. Theorem 2.2.2 in \cite{fallat} says that every invertible totally nonnegative matrix can be written as $(L_n\left(l_{k}\right)L_{n-1}\left(l_{k-1}\right)\cdots L_2\left(l_{k-n+2}\right))\\ (L_n\left(l_{k-n+1}\right)L_{n-1}\left(l_{k-n}\right) \cdots L_3\left(l_{k-2n+4}\right))\cdots(L_n\left(l_{1}\right))D (U_n\left(u_{1}\right)) (U_{n-1}\left(u_{2}\right) U_n\left(u_{3}\right))\cdots  (U_2\left(u_{k-n+2}\right)\cdots \\  U_{n-1}\left(u_{k-1}\right)U_n\left(u_{k}\right)),$ where $k=\binom{n}{2},\  l_i,u_j\geq 0$ for all $i,j\in\{1,2,\ldots,k\}$ and 
	$D=\text{diag}([d_i])$
	is a diagonal matrix with all $d_i >0$. The below  proposition gives that $u_j=0$ for $\mathfrak{s}$, which is not surprising in view of \cite[Theorem 7]{johnson99}.

\begin{pro}\label{stirling}
  The $n \times n$ unsigned Stirling matrix of first kind $\mathfrak{s}$ can be factorized as
\begin{equation}
 \mathfrak{s} = (L_{n}({n-1})L_{n-1}({n-2})\cdots L_{2}({1})) (L_{n}({n-2})L_{n-1}({n-3})\cdots L_{3}({1}))\cdots(L_{n}({2})L_{n-1}({1}))(L_{n}({1})).
\end{equation}
\end{pro}
\begin{proof}
Since $(L_{i}(s))^{-1} = L_{i}(-s)$, so it is enough to show that
\begin{eqnarray} 
    & &\left(L_{n}({-1})\right)\ \left(L_{n-1}({-1})L_{n}({-2})\right)\cdots\left(L_{3}({-1})\cdots L_{n-1}({-(n-3)})L_{n}({-(n-2)})\right) \left(L_{2}({-1})\cdots\right.\nonumber\\ & &\hspace{8cm}\left. L_{n-1}({-(n-2)})L_{n}({-(n-1)})\right)\mathfrak{s} = I_n, \label{eqn:1}
\end{eqnarray}
where $I_n$ is the $n\times n$ identity matrix.
We shall prove \eqref{eqn:1} by induction on $n$. For clarity, we shall denote the $n\times n$ matrices $\mathfrak s$ and $L_i(s)$, by $\mathfrak s_n$ and $L_i(s)_{(n)}$, respectively. For $n=2$, $\mathfrak{s}_2=\begin{bmatrix}1 & 0\\1&1\end{bmatrix}$, which is clearly equal to $L_2(1)_{(2)}$. Let us assume that \eqref{eqn:1} holds for $n$.
We have the following recurrence relation \cite[p. 166]{Henry} for $\mathfrak s_{ij}$ :
\begin{eqnarray}
&&\mathfrak s_{00} = 1, \mathfrak s_{0j} = 0, \mathfrak s_{i0} = 0\nonumber\\
&&\mathfrak s_{i+1,j} = \mathfrak s_{i,j-1}+i\ \mathfrak s_{ij}.\label{recurrence}
\end{eqnarray}
So for $1\leq k\leq n$, 
multiplying $\mathfrak{s}_{n+1}$ on the left by $(L_{k+1}({-k}))_{(n+1)}$  replaces its $(k+1)$th row by the row whose first element is $0$ and $j$th element is the $(j-1)$th element of the previous row. 
So we get            
\begin{equation} \label{eqn:2}
    \left(L_{2}({-1})_{(n+1)}\cdots L_{n}({-(n-1)})_{(n+1)}L_{n+1}({-n})_{(n+1)}\right)\mathfrak{s}_{n+1} = \begin{bmatrix}
    1&0\\
    0&\mathfrak{s}_{n}
    \end{bmatrix}.
\end{equation}
 It is easy to see that 
\begin{equation*}
    L_{i}({s})_{(n+1)} = \begin{bmatrix}
        1 & 0\\
        0 & L_{i-1}({s})_{(n)}
    \end{bmatrix}.
\end{equation*}
Hence 
\begin{align*} 
   & \left(L_{n+1}({-1})_{(n+1)}\right)\left(L_{n}({-1})_{(n+1)}L_{n+1}({-2})_{(n+1)}\right)\cdots\left(L_{3}({-1})_{(n+1)}\cdots L_{n}({-(n-2)})_{(n+1)}L_{n+1}({-(n-1)})_{(n+1)}\right) \\
   &\hspace{0.2cm}= \begin{bmatrix}
        1 & 0\\
        0 & \left(L_{n}({-1})_{(n)}\right)\left(L_{n-1}({-1})_{(n)}L_{n}({-2})_{(n)}\right)\cdots\left(L_{2}({-1})_{(n)}\cdots L_{n-1}({-(n-2)})_{(n)}L_{n}({-(n-1)})_{(n)}\right)
    \end{bmatrix}.
\end{align*} Using induction hypothesis and \eqref{eqn:2}, we obtain 
\begin{align*}
    &\left(L_{n+1}({-1})_{(n+1)}\right)\left(L_{n}({-1})_{(n+1)}L_{n+1}({-2})_{(n+1)}\right)\cdots\left(L_{3}({-1})_{(n+1)}\cdots L_{n}({-(n-2)})_{(n+1)}L_{n+1}({-(n-1)})_{(n+1)}\right)\\
    &\left(L_{2}({-1})_{(n+1)}\cdots L_{n}({-(n-1)})_{(n+1)} L_{n+1}({-n})_{(n+1)}\right)\mathfrak{s}_{n+1} = I_{n+1}.
\end{align*}
\end{proof}

As an immediate consequence, we obtain the following.
\begin{theorem}
The symmetrized unsigned Stirling matrix of first kind $\mathfrak s \mathfrak s^*$ is totally positive.
\end{theorem}
\begin{proof}
 This  follows from the above Proposition \ref{stirling} and Corollary 2.2.3 of \cite{fallat}.
\end{proof}

We now show that symmetrized Stirling matrix of second kind $\mathcal S \mathcal S^*$ is totally positive. For that we first prove that $\mathcal S$ is triangular totally positive.  For $\alpha=\{\alpha_1,\ldots,\alpha_p\}$, $ \gamma =\{\gamma_{1}, \ldots, \gamma_{p}\}$ with $1 \leq \alpha_{1} < \cdots < \alpha_{p} \leq n$ and $1 \leq \gamma_{1}< \cdots < \gamma_{p} \leq n$, let $A[\alpha,\gamma]$ denotes the submatrix of  $A$ obtained by picking rows $\alpha_1,\ldots,\alpha_p$ and columns $\gamma_1,\ldots,\gamma_p$ of $A$. The \textit{dispersion} of  $\alpha$, denoted by $d(\alpha)$, is defined as $d(\alpha) = \alpha_{p}-\alpha_{1}-(p-1)$.  Note that $d(\alpha)=0$ if and only if $\alpha_1,\alpha_2,\ldots,\alpha_p$ are consecutive $p$ numbers. We denote by $\alpha'$ the set $\{\alpha_{1}+1,\alpha_{2}+1, \ldots, \alpha_{p}+1\}$, and by $\sigma^{(p)}$ the set $\{1,\ldots,p\}$.
\begin{pro}\label{Stirling II}
The Stirling matrix of second kind $\mathcal S$ is triangular totally positive.
\end{pro}
 \begin{proof} 
 By Theorem 3.1 of \cite{colin}, $\mathcal S$ is triangular totally positive if and only if $\det (\mathcal S[\alpha, \sigma^{(p)}])>0$ for all $1\leq p \leq n$ and for all $\alpha=\{\alpha_1,\ldots,\alpha_p\}$ satisfying $1 \leq \alpha_{1} < \cdots < \alpha_{p} \leq n$ and $d(\alpha) = 0$. For $p = n$, this is obviously true. Let $1\leq p < n$.
 If $\alpha_{1}=1$ and $d(\alpha)=0$, then $\alpha=\sigma^p$ and $\det(\mathcal S[\sigma^{(p)},\sigma^{(p)}]) = 1 >0$. Next we show that if $\det (\mathcal S[\alpha, \sigma^{(p)}])>0$,
where $d(\alpha)=0$, then $\det (\mathcal S[\alpha', \sigma^{(p)}])>0$ (and $d(\alpha')=0$). 
 
 Let $T=[t_{ij}]$ be defined as $t_{ij} = \begin{cases}
& i \quad \text{if $i = j$}\\
& 1 \quad \text{if $j-i = 1$}\\
& 0 \quad \text{otherwise}
\end{cases}$.
 We prove that \begin{equation}
    \mathcal S[\alpha',\sigma^{(p)}]=\mathcal S[\alpha, \sigma^{(p)}]T  \text{ for } 1 \leq p < n. \label{st}
 \end{equation} The $(i,j)${th} entry of $\mathcal S[\alpha, \sigma^{(p)}]T$ is $j\, S(\alpha_{i},j)+S(\alpha_{i},j-1)$.
 The Stirling numbers of second kind satisfy the following recurrence relation: $$S(0,0) = 1;$$ $$ \text{for } \ell,m\geq 1,\ S(0,m)=0=S(\ell,0), \ S(\ell,m) = m\,S(\ell-1,m)+S(\ell-1,m-1).$$  
Thus the $(i,j)${th} entry of $\mathcal S[\alpha, \sigma^{(p)}]T$ is $= S(\alpha_{i}+1,j)$, which is also the $(i,j)${th} entry of $\mathcal S[\alpha', \sigma^{(p)}]$. Hence \eqref{st} holds, which gives that $ \det(\mathcal S[\alpha',\sigma^{(p)}])=p! \det (\mathcal S[\alpha, \sigma^{(p)}]) >0$. 
\end{proof}


\begin{theorem}
The symmetrized Stirling matrix of second kind $\mathcal S \mathcal S^*$ is totally positive.
\end{theorem}
\begin{proof}
This follows from Proposition \ref{Stirling II} and Corollary 2.4.2 of \cite{fallat}.
\end{proof}

Next, we show that $\mathfrak B= [\mathfrak w({i+j})]$ is totally positive. Let $Y = \left(y_{ij}\right)_{i,j=0}^{n-2}$ be the lower triangular  matrix defined recursively by $y_{00} = 1,\ y_{0j}= 0 \text{ for } j>0, \text{ and }y_{ij} = y_{i-1,j-1}+(j+2)y_{i-1,j}+(j+1)y_{i-1,j+1} \text{ for } i\geq 1,$ where $y_{i,-1}=0$ for every $i$ and $y_{in}=0$ for $0\leq i\leq n-3$.

\begin{theorem}
 The matrix $ \mathfrak{B} = [\mathfrak w({i+j})] $ is totally positive.                                                                            
\end{theorem}                                       
\begin{proof}                                       
Let $\mathfrak{B}(n,1) = [\mathfrak{w}(i+j+1)]_{i,j=0}^{n-2}$ be the matrix obtained from $\mathfrak B$ by deleting its first column and $n$th row.
From the proof of the Theorem in  \cite{Martin}, $\mathfrak{B} = LL^*$, where $L=X \ \text{diag} \left(\left[\sqrt{i!}\right]\right)_{i=0}^{n-1}$, and  $\mathfrak{B}(n,1) = L' L'^*$, where $L'=Y \ \text{diag} \left(\left[\sqrt{i!)}\right]\right)_{i=0}^{n-2}$. Hence $\mathfrak{B}$ and $\mathfrak{B}(n,1)$ are positive semidefinite. The  Theorem in \cite{Martin} also shows that both $\mathfrak B$ and $\mathfrak B(n,1)$ are nonsingular, hence they are positive definite. Since $\mathfrak B$ is a Hankel matrix, the result now follows from Theorem 4.4 of \cite{pinkus}.
 
\end{proof}

Now we show that $\mathfrak B$ is infinitely divisible only upto order 4.  For $A=[a_{ij}]$, let $\log A = [\log(a_{ij})].$ Let $\Delta A$ denote the $(n-1)\times (n-1)$ matrix $[a_{ij}+a_{i+1,j+1}-a_{i+1,j}-a_{i,j+1}]_{i,j=1}^{n-1}$.


\begin{theorem}
The $n \times n$  matrix $\mathfrak{B}$ is infinitely divisible if and only if $ n \leq 4 $. 
\end{theorem}

\begin{proof}
We denote the $n\times n$ matrix $\mathfrak B$ by $\mathfrak B_n$. Since $\mathfrak B_n$  is a principal submatrix of  $\mathfrak B_{n+1}$, it is enough to show that $\mathfrak B_4$  is infinitely divisible but $\mathfrak B_5$ is not infinitely divisible.

 By Corollary 1.6 and Theorem 1.10 of \cite{horn}, to prove infinite divisibility of $\mathfrak B_4$, it is enough to prove that $\Delta \log \mathfrak B_{4} $ is positive definite. 
Now \begin{eqnarray*}& \mathfrak{B}_{4}=	
\begin{bmatrix}
1&1&2&5\\1&2&5&15\\2&5&15&52\\5&15&52&203 
\end{bmatrix},\ \log \mathfrak B_{4} = \begin{bmatrix}
0&0&\log 2&\log 5\\0&\log 2&\log 5&\log 15\\\log 2&\log 5&\log 15&\log 52\\\log 5&\log 15&\log 52&\log 203
\end{bmatrix}, \\ & \text{ and } \Delta \log \mathfrak B_{4} = \begin{bmatrix}
\log 2&\log (5/4)&\ \log(6/5)\\ \log(5/4)&\log(6/5)&\log(52/45)\\\log(6/5)&\log(52/45)&\log(3045/2704)
\end{bmatrix}.
\end{eqnarray*}

Since all the leading principal minors of $\Delta \log  \mathfrak B_{4}$ are positive, we get that $\Delta\log \mathfrak B_{4}$ is positive definite. Hence $\mathfrak B_4$ is infinitely divisible.
Since $\det\left(\mathfrak{B}_5^{\circ (\frac{1}{4})}\right) = -1.62352 \times 10^{-9} <0$, $\mathfrak B_5$ is not infinitely divisible. 
\end{proof}

\textbf{Acknowledgement} {\it The possibility of studying these problems was suggested to us by Rajendra Bhatia. It is a pleasure to record our thanks to him. We are also very thankful to the anonymous referee for several helpful comments.}

\end{document}